\documentclass[a4paper,10pt]{article}
\usepackage{pgf,tikz}
\usetikzlibrary{arrows}
\usepackage{geometry}
\usepackage{latexsym}
\usepackage{amsfonts, amsmath, amssymb, amsthm}
\usepackage[utf8]{inputenc}
\usepackage[T1]{fontenc}
\usepackage[english]{babel}
\usepackage{stmaryrd}
\usepackage{subfig}
\usepackage{frcursive}
\usepackage{mathrsfs}
\usepackage{enumitem}
\usepackage{yhmath}

\newcommand{\N}{\mathbb{N}}
\newcommand{\Z}{\mathbb{Z}}

\newcommand{\R}{\mathbb{R}}
\newcommand{\C}{\mathbb{C}}

\newcommand{\tiling}{\mathcal{T}}

\newcommand{\dominov}[2]{\draw (#1,#2)--++(1,0)--++(0,2)--++(-1,0)--++(0,-2);
\fill[gray](#1,#2)--++(1,0)--++(0,2)--++(-1,0)--++(0,-2);}
\newcommand{\dominoh}[2]{\draw (#1,#2)--++(2,0)--++(0,1)--++(-2,0)--++(0,-1);}

\newcommand{\patchv}[2]{\dominoh{#1}{#2}\dominoh{#1}{#2+3}\dominov{#1}{#2+1}\dominov{#1+1}{#2+1}}
\newcommand{\patchh}[2]{\dominoh{#1+1}{#2+1}\dominoh{#1+1}{#2}\dominov{#1}{#2}\dominov{#1+3}{#2}}

\newcommand{\patchvbis}[2]{\patchh{#1}{#2}\patchh{#1}{#2+6}\patchv{#1}{#2+2}\patchv{#1+2}{#2+2}}
\newcommand{\patchhbis}[2]{\patchh{#1+2}{#2+2}\patchh{#1+2}{#2}\patchv{#1}{#2}\patchv{#1+6}{#2}}

\newcommand{\patchhter}[2]{\patchhbis{#1+4}{#2+4}\patchhbis{#1+4}{#2}\patchvbis{#1}{#2}\patchvbis{#1+12}{#2}}

\newtheorem{theorem}{Theorem}[section]
\newtheorem{lemma}[theorem]{Lemma}

\newtheorem{definition}[theorem]{Definition}

\theoremstyle{definition}
\newtheorem{remark}[theorem]{Remark}
\newtheorem{example}[theorem]{Example}

\title{Spectral measure at zero for self-similar tilings}
\author{Jordan Emme \thanks{Aix-Marseille Université, CNRS, Centrale Marseille, I2M, UMR 7373, 13453 Marseille, France. E-mail: jordan.emme@univ-amu.fr}}
\date{}
\begin{document}

\maketitle
\begin{abstract}The goal of this paper is to study the action of the group of translations over self-similar tilings in the euclidian space $\R^d$. It investigates the behaviour near zero of spectral measures for such dynamical systems. Namely, the paper gives a Hölder asymptotic expansion near zero for these spectral measures. It is a generalization to higher dimension of a result by Bufetov and Solomyak who studied self similar-suspension flows for substitutions in \cite{modcont}. The study of such asymptotics mostly involves the understanding of the deviations of some ergodic averages.
\end{abstract}

\section*{Keywords}

Self-similar tilings, ergodic theory, spectral measures.

\section*{Mathematic Subject Classification}

37B50  	Multi-dimensional shifts of finite type, tiling dynamics\\
37A30  	Ergodic theorems, spectral theory, Markov operators

\section{Introduction}

Spectral analysis plays an important role in dynamical systems and ergodic theory. Some properties of the spectrum of the Koopman operator, or of a group of operators acting on a space, translate into dynamical properties (in our case, the group $\R^d$ acts on the space of tilings and this defines a group of operators acting on the functions of the tiling space). In this light, spectral measures hold some amount of information concerning dynamical systems. Spectral analysis of some dynamical systems via spectral measures was done in \cite{queffelec} for instance.

In this paper we are interested in a particular family of dynamical systems which is extensively studied: those who arise from substitutions or self-similar tilings. Substitutions are essentially combinatorial objects and have no geometry a priori. However it is possible to interpret infinite words as tilings of the real line (or real semiline depending in which context we are working) and thus have some (limited) geometry. Self-similar tilings are objects that arise when looking for a natural generalization of this in higher dimension.

Bufetov and Solomyak have studied the deviation of ergodic averages for the action of $\R^d$ by translation over self-similar tilings in \cite{limitthm}. Such results generalize famous works by Bufetov for $d=1$ in \cite{finit_add_measure} or \cite{bufetov_annals}.

The precise study of the deviation of ergodic averages allows the understanding of the modulus of continuity for suspension flows over substitution dynamical systems in \cite{modcont}. One of the results is that there is a Hölder asymptotic expansion on balls centered at zero for the spectral measure of self-similar suspension flows (which we could see as self similar tilings of the real line). The Hölder exponent is explicitly computed.

The aim of this paper is to give a generalization of this result for spectral measures of self-similar tilings for $d\geq1$. We prove that there is always a Hölder asymptotic expansion regardless of the choice of $d$, and give an explicit computation of the Hölder exponent.

Section 2 is dedicated to defining self-similar tilings and their associated dynamical systems along with some basic properties.

Section 3 exposes some of the results that were obtained in \cite{limitthm,modcont} which are crucial to the proof of our main theorem. We also give the principal definition of the paper in this section: that of spectral measures.

In Section 4 we state the main result and give its proof based on the study of   the finitely-additive measures from \cite{limitthm} and the asymptotics of the ergodic integrals.

\subsection{Result}

In this article we study a particular class of tilings of the euclidian space $\R^d$ obtained by the following process. We fix a finite amount of tiles called prototiles. We endow the set of tilings of $\R^d$ by the prototiles with a topology detailed in Section 2. We choose a real constant $\lambda$ and consider the images of these prototiles by the homothety of factor $\lambda$. If we can tile these images with translates of the prototiles, then doing so defines a substitution rule $\zeta$. If this rule $\zeta$ satisfies certain conditions, then iterating it over a prototile spans bigger and bigger patches (i.e. union of tiles) which eventually tile the whole space. Taking the closure of the orbit of this tiling via the action of $\R^d$ by translation defines a space $\mathfrak{X}_{\zeta}$. Some  other conditions on $\zeta$ ensure that the dynamical system $(\mathfrak{X}_{\zeta}, \R^d)$ is uniquely ergodic and we denote its unique ergodic probability measure by $\mu$. Then to any real valued function $f$ from $\mathfrak{X}_{\zeta}$ to $\R$ in $L^1(\mathfrak{X}_{\zeta}, \mu)$, one can associate a spectral measure $\sigma_f$. The following theorem states that, after renormalisation, this measure is a Radon measure on balls centered at zero. This can be seen as a Hölder regularity property. The definitions of the technical conditions of this theorem \--- such as cylindrical functions and the quantity $m_{\Phi_v^-}(f)$ \---  are given respectively in  Definitions \ref{d.cylindrical} and \ref{d.mphimoins}.

\begin{theorem}\label{t.main}
 Let $\mathcal{A}=\{T_1,...,T_m\}$ be a set of polyhedral prototiles in $\R^d$ and $\zeta$ a primitive non-periodic tile-substitution over those prototiles defining a self-similar tiling with finite local complexity. Let $\lambda$ be the real expansion constant of $\zeta$, and $S$ be its incidence matrix with eigenvalues
 $$\theta_1 > \theta_2 > |\theta_3| \geq ... \geq |\theta_m|$$
 satisfying $\theta_2 > \theta_1^{\frac{d-1}{d}}$. Let $f$ be a cylindrical function with zero mean $\int f d\mu =0$ and $ m_{\Phi_v^-}(f) \neq 0$ (with $v$ being in the eigenspace associated to $\theta_2$). Let $\sigma_f$ be its associated spectral measure on $\R^d$. Then there exists a non-trivial positive $\sigma$-finite Radon measure $\eta$ on $\R_+$ such that:
 
 $$\lim_{N\rightarrow\infty} \frac{\sigma_f (B(0,a\lambda^{-N}))}{\lambda^{-N(2d - 2\alpha)}}=\eta([0,a]),\quad \text{for all $a>0$ such that $\eta(\{a\})=0$,}$$
 where 
 
 $$\alpha = \frac{d\log(\theta_2)}{\log(\theta_1)} \in (d-1,d).$$
\end{theorem}

This theorem is a natural generalization of \cite[Theorem 6.2]{modcont} which gives the same behaviour for $d=1$. 

Remark that in this theorem, we have:
$$
\theta_1=\lambda^d.
$$

\subsection{Outline of the paper}

Section \ref{s.tilings} is devoted to defining the framework. More precisely, we give the definition of self-similar tilings. We recall some famous results. In particular, that a primitive tile substitution $\zeta$ with finite local complexity defines a uniquely ergodic dynamical system $(\mathfrak{X}_{\zeta},\R^d)$, where $\mathfrak{X}_{\zeta}$ is the set of tilings of the euclidian space $\R^d$ obtained with the substitution rules. We also recall how to subdivide a certain tiling using the implicit substitution rules of $\zeta$.

With these subdivisions in mind, we recall in Section \ref{s.deviation} the construction of finitely-additive measures. For a thorough and detailed construction of such objects, we refer the reader to \cite{finit_add_measure}. These are capital for the understanding of ergodic deviations. Namely, we state the main theorem of \cite{limitthm} which gives a precise behaviour of Birkhoff sums depending on a decomposition of $\R^d$ in stable subspaces for the transpose of the incidence matrix of the tile-substitution $\zeta$.

Section \ref{s.spectral_measures} is devoted to proving Theorem \ref{t.main}. We start by stating the spectral theorem which allows us to define the spectral measures. The proof of the main theorem is then divided into 5 steps. In the first step we integrate a test function (which approximates the characteristic function of an interval) with respect to the spectral measure. The spectral isomorphism allows to express our integral in terms of ergodic sums on balls of radius $R$. The second step uses Theorem \ref{t.erg_dev} on ergodic deviations in order to estimate the rest of our ergodic sum as $R$ goes to infinity. Step 3 is just a change of variables in the main term, using the 'self-similarity' properties of the finitely-additive measures. We can then compute the limit as $R$ goes to infinity. Step 4 is devoted to the regularity of this limit and how to extend it to a bilinear form, right continuous on continuous compactly-supported functions. Step 5 introduces a technical lemma which allows us to make the link between our limit and distribution and then with Radon measures.

\subsection{Acknowledgements}
I would like to thank \textsc{Alexander Bufetov} and \textsc{Boris Solomyak} for asking me the question that this paper answers and for their very helpful advice and the enlightening discussions we had about this problem.

\section{Self-similar tilings}\label{s.tilings}

In all that follows, we are working in the euclidian space $\R^d$. We are interested in tiling the euclidian space $\R^d$ (which comes with its origin) with tiles which are compact, and the closure of their interior.  In particular, we are interested in a class of aperiodic tilings called self-similar tilings. They can be obtained by iteration of a substitution rule over tiles. A famous class of such tilings is found in \cite{penrose_pentaplexity}. We give an example of such a tiling in Figure \ref{f.tiling} before properly defining this notion. We almost always use the case of Figure \ref{f.tiling} to illustrate the numerous formal definitions in this section.

 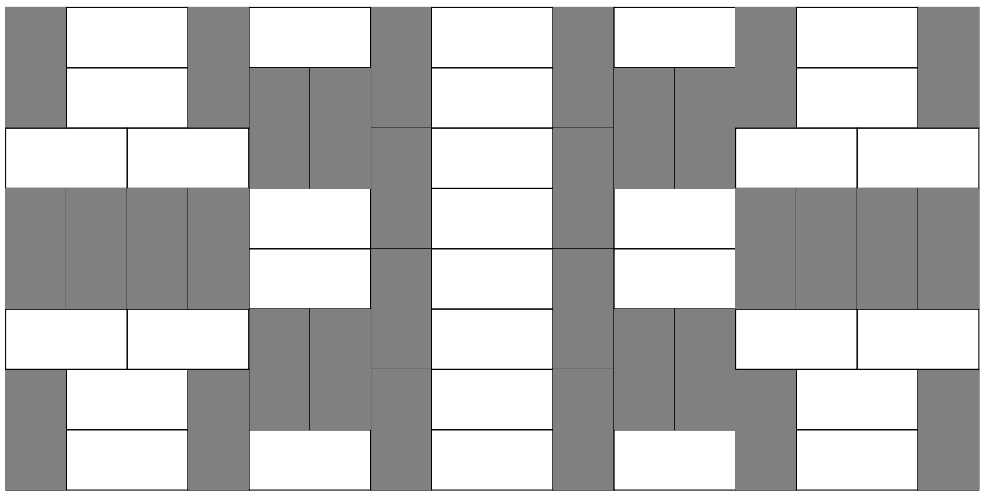
\begin{figure}[!ht]
\center
\begin{tikzpicture}[scale=0.8]
\patchhter{0}{0}
\end{tikzpicture}
\caption{Part of a self-similar tiling: the table tiling.}\label{f.tiling}
\end{figure}

Let us now give all the necessary formalisms.

\begin{definition}
Fix a set of types (sometimes called colours) $\{1,...,m\} \subset \N.$ A tile $T$ in $\R^d$ is a pair $(A,i)$ where $A=:\mathrm{supp}(T)$ (the support of $T$) is a closed compact subset of $\R^d$ such that $A=\overline{\mathring{A}}$ and $i\in\{1,...,m\}$ is the type of the tile.

We denote the type of the tile $T$ by $l(T).$
\end{definition}

\begin{definition}
Define the translate of a tile $T$ by a vector $x \in \R^d$:
$$
T+x:=\left(\mathrm{supp}(T)+x,l(T)\right).
$$
\end{definition}

Usually, we fix an `alphabet': a finite set of tiles $\mathcal{A}=\{T_1,...,T_m\}$. It is implied that the tile $T_i$ has type $i$ (but two distinct tiles in this alphabet do not necessarily have different support). The tiles $T_i$ are called prototiles.

After having fixed a finite set of prototiles we are interested in tilings of the euclidian space only with translates of these prototiles. To be formal, a tiling $\mathcal{T}$ of $\R^d$ by tiles in $\mathcal{A}=\{T_1,...,T_m\}$ is given by a covering of $\R^d$ in the following way:

$$
\R^d=\bigcup_{v \in V} \left[ \mathrm{supp}(T_{i_v})+v \right]
$$
and $\left(\mathring{\wideparen{\mathrm{supp}(T_{i_v})}}+v\right)\cap \left( \mathring{\wideparen{\mathrm{supp}(T_{i_w})}}+w\right)= \emptyset$ whenever $v\neq w$.

Here $V$ denotes a countable set of vectors in $\R^d$ and $i_v$ is in $\{1,...,m\}$ for any $v$ in $V$.

In other words, a tiling of the space with tiles in $\mathcal{A}$ is a covering of this space with translates of the prototiles in $\mathcal{A}$ in such a way that only the boundaries of the support of the tiles intersect.

In the case of Figure \ref{f.tiling}, a possible choice of prototiles is: $$\left\{\left( [-\frac12,\frac12]\times[-1,1],1 \right), \left( [-1,1]\times[-\frac12,\frac12],2 \right)\right\}.$$ For visual purposes, the tiles of type $1$ are coloured in grey and the tiles of type $2$ are coloured in white.

We now define patches.

\begin{definition}
 A patch $P$ is a finite union of tiles of disjoint interiors. The support of a patch $P$ is the set $\mathrm{supp}(P)=\displaystyle\bigcup_{T \in P}\mathrm{supp}(T)$. A patch $P$ has diameter at most $R$ if its support is contained in a ball of diameter $R$. Denote by $\mathcal{P}_{\mathcal{A}}$ the set of patches whose tiles are translates of prototiles in $\mathcal{A}$.
\end{definition}

Let us define tile substitutions.

\begin{definition}
 Let $\phi: \R^d \rightarrow \R^d$ be an homothety with expansion constant $\lambda >1$. A map $\zeta:\mathcal{A} \rightarrow \mathcal{P}_{\mathcal{A}}$, where $\mathcal{P}_{\mathcal{A}}$ denotes the set of patches with tiles in $\mathcal{A}$, is called a tile substitution with expansion $\phi$ if $ \mathrm{supp}(\zeta(T_i)) = \phi(\mathrm{supp}(T_i))$.
\end{definition}

\begin{example}\label{e.zeta}
 An example of tile substitutions using tiles from Figure \ref{f.tiling} is:
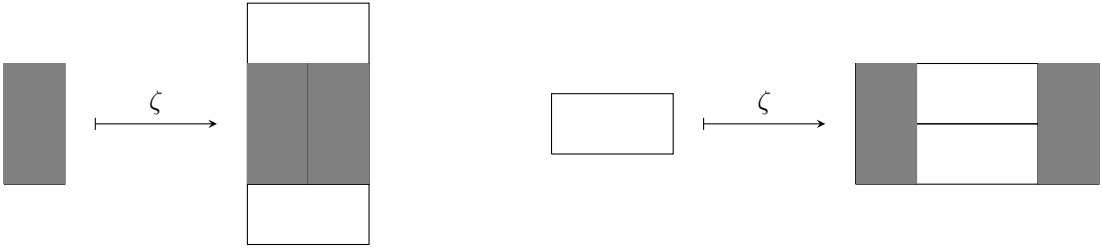
\begin{center}
 \begin{tikzpicture}[scale=0.8]
 \dominov{0}{0}
 \patchv{4}{-1}
 \dominoh{9}{0.5}
 \patchh{14}{-0}
  \draw[->][>=stealth] (1.5,1)--(3.5,1) node[midway,above]{$\zeta$} ;
 \draw (1.5,0.9)--(1.5,1.1);
  \draw[->][>=stealth] (11.5,1)--(13.5,1) node[midway,above]{$\zeta$} ;
 \draw (11.5,0.9)--(11.5,1.1);
\end{tikzpicture}
\end{center}
\end{example}

Notice that this substitution rule has expansion constant $\lambda=2$.

\begin{definition}
   For any tile substitution $\zeta$, define a substitution matrix $S$ such that $S_{i,j}$ is the number of tiles $T_i$ appearing in the patch $\zeta(T_j)$. 
\end{definition}

In the example given previously, the substitution matrix is $S=\begin{pmatrix}
  2&2\\
  2&2
\end{pmatrix}.$

 \begin{remark}\label{r.theta_1lambda}
   If $\theta_1$ denotes the Perron-Frobenius eigenvalue of $S$, we have $\theta_1 = \lambda^d$. We refer the reader to \cite{robinson} for more details.
 \end{remark}

 Remark that the Perron Frobenius eigenvalue of $\begin{pmatrix}
  2&2\\
  2&2
\end{pmatrix}$ is indeed $\theta=4=\lambda^2.$ 

\begin{remark}\label{r.primitive}
 In all that follows, we only study primitive tile substitutions, \textit{i.e.} tile-substitutions whose matrix $S$ satisfies the property that there exists a positive integer $k$ such that $S^k$ has only positive entries.
\end{remark}

We extend the definition of $\zeta$ to $\mathcal{P}_{\mathcal{A}}$ and to the set of tilings of $\R^d$ by tiles in $\mathcal{A}$.
 
 In order to extend the definition to the set of patches $\mathcal{P}_{\mathcal{A}}$, we first define the application $\zeta$ with real expansion constant $\lambda$ on the set of tiles which are translates of prototiles in $\mathcal{A}$. Let $T$ be a translate of prototile $T_i$ for a certain $i$ by a vector $x \in \R^d$, i.e. $T=T_i+x$. Then we define $\zeta(T)=\zeta(T_i) +\lambda x$. Now in order to define $\zeta$ on the set of patches we proceed as follows.
 
 Let $P$ be a patch in $\mathcal{P}_{\mathcal{A}}$. There exists a finite set of vectors $\{v_1,...,v_n\}$ in $\R^d$ and a finite set of integers $\{i_1,...,i_n\}$ in $\{1,...,m\}$ such that
 $$
 P=\bigcup_{j=1}^m T_{i_j}+v_j.
 $$
 Then define the patch $\zeta(P)$ in the following manner
 $$
  \zeta(P)=\bigcup_{j=1}^m \zeta(T_{i_j})+\phi(v_j).
 $$
 
 We extend the definition of $\zeta$ to a whole tiling of $\R^d$ in the same way.
 
 \begin{example}
  For instance, expanding the definition of the substitution given in Example \ref{e.zeta} to patches is illustrated in Figure \ref{f.subs_on_patches}.
  
  \begin{figure}[!ht]
 \begin{tikzpicture}[scale=0.8]
 \dominov{0}{0}
 \dominov10
 \patchv{4}{-1}
 \patchv{6}{-1}
 \dominoh{10}{1}
  \dominoh{10}{0}
 \patchh{14}{1}
  \patchh{14}{-1}
  \draw[->][>=stealth] (2.5,1)--(3.5,1) node[midway,above]{$\zeta$} ;
 \draw (2.5,0.9)--(2.5,1.1);
  \draw[->][>=stealth] (12.5,1)--(13.5,1) node[midway,above]{$\zeta$} ;
 \draw (12.5,0.9)--(12.5,1.1);
  \draw[very thick] (0,0)--++(1,0)--++(0,2)--++(-1,0)--++(0,-2);
  \draw[very thick] (1,0)--++(1,0)--++(0,2)--++(-1,0)--++(0,-2);
 \draw[very thick] (10,1)--++(2,0)--++(0,1)--++(-2,0)--++(0,-1);
  \draw[very thick] (10,0)--++(2,0)--++(0,1)--++(-2,0)--++(0,-1);
  
 \draw[very thick] (4,-1)--++(2,0)--++(0,4)--++(-2,0)--++(0,-4);
  \draw[very thick] (6,-1)--++(2,0)--++(0,4)--++(-2,0)--++(0,-4);
 \draw[very thick] (14,1)--++(4,0)--++(0,2)--++(-4,0)--++(0,-2);
  \draw[very thick] (14,-1)--++(4,0)--++(0,2)--++(-4,0)--++(0,-2);
\end{tikzpicture}\caption{Substitution rule on some patches}\label{f.subs_on_patches}
  \end{figure}
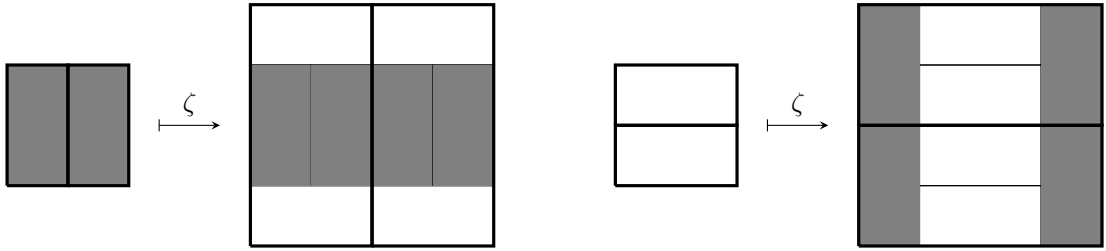
 \end{example}

 \begin{definition}
  A tiling $\tiling$ of $\R^d$ is self similar if $\zeta(\tiling) =\tiling$ for a primitive tile substitution $\zeta$.
 \end{definition}

%
%
%


 \begin{definition}
Given a tile-substitution $\zeta$, we define the tiling space $\mathfrak{X}_{\zeta}$:\\
A tiling $\tiling$ of $\R^d$ is in $\mathfrak{X}_{\zeta}$ if every finite patch of $\tiling$ is a sub-patch of $\zeta^k(\tiling_i)$ for some $k \in \N$ and some $i \in \llbracket1,m\rrbracket$.
 \end{definition}

 We endow this set with a topology by defining the metric $d$:
 
 let $\mathcal{T}_1$ and $\mathcal{T}_2$  be two tilings in $\mathfrak{X}_{\zeta}$. We define $d(\mathcal{T}_1,\mathcal{T}_2)=\min\left( \frac{1}{\sqrt2}, \epsilon_0 \right)$ where $\epsilon_0$ is the infimum of the set of $\epsilon>0$ which satisfy the following condition: there exists a vector $v$  in $\R^d$ such that $\|v\|<\epsilon$ (where $\|v\|$ is the euclidian norm) and the biggest patches of $\mathcal{T}_1$ and $\mathcal{T}_2+v$ inside the ball $B(0,1/\epsilon)$ are the same.

  In other words, two tilings are close if, up to a small translation, one can see the same pattern in a large ball centered at the origin. For more details on this distance and topology the reader can refer to \cite{robinson}.
 
 \begin{definition}
  Let $\mathcal{A}=\{T_1,...,T_m\}$ be an alphabet. Let $\zeta$ be a tile substitution with expansion map $\phi$. A tile $T$ is said to be a tile of order $k\in \Z$ if there exist $i \in \llbracket1,m\rrbracket$ and $y \in \R^d$ such that $\mathrm{supp}(T)-y=\phi^k(\mathrm{supp}(T_i))$
 \end{definition}
 
 \begin{example}
  We draw in thick black lines the boundary of tiles of order 3 in the example of tiling given in Figure \ref{f.tiling}.
 \begin{figure}[!ht]
\center
\begin{tikzpicture}[scale=0.8]
\patchhter{0}{0}
 \draw[ultra thick] (0,0)--++(4,0)--++(0,8)--++(-4,0)--++(0,-8);
  \draw[ultra thick] (12,0)--++(4,0)--++(0,8)--++(-4,0)--++(0,-8);
 \draw[ultra thick] (4,0)--++(8,0)--++(0,4)--++(-8,0)--++(0,-4);
  \draw[ultra thick] (4,4)--++(8,0)--++(0,4)--++(-8,0)--++(0,-4);
\end{tikzpicture}
\end{figure}

 \end{example}

 \begin{definition}
  We extend the definition of the tile-substitution $\zeta$ to tiles of order $k \in \Z$ and denote it $\phi^k \zeta$:
  $$
  \phi^k \mathcal{A}:=\{(\phi^k\mathrm{supp}(T_i),i)\}_{i \in \llbracket1,m\rrbracket},
  $$  
  $$
  \forall i \in \llbracket1,m\rrbracket,\ (\phi^k \zeta)((\phi^k(\mathrm{supp}(T_i)),i)):=\phi^k(\mathrm{supp}(\zeta(T_i))).
  $$
 \end{definition}

This allows to define the subdivision map:
$$
\Upsilon_k:\mathfrak{X}_{\phi^k\zeta}\mapsto\mathfrak{X}_{\phi^{k-1}\zeta}.
$$
 This divides every tile of order $k$ into sub-tiles according to the substitution rule. Note that is  well defined if and only if the substitution is non-periodic (that is, if $\R^d$ acts freely over $\mathfrak{X}_\zeta$) as justified by the following theorem.
 
 \begin{theorem}[\cite{solomyak_injectivity}]
  The map $\zeta:\mathfrak{X}_{\zeta}\rightarrow\mathfrak{X}_{\zeta}$ is injective if and only if $\zeta$ is a non-periodic substitution.
 \end{theorem}

 With this division, we define inductively, for any  tiling $\mathcal{T} \in \mathfrak{X}_{\zeta}$ and for any integer $k$, the subdivided tiling $\mathcal{T}^{(k)}$ in the following way:
$$
\mathcal{T}^{(0)}:=\mathcal{T},
$$

$$
\mathcal{T}^{(k)}=\left\{\begin{array}{ccc}
                          \Upsilon_k^{-1}...\Upsilon_1^{-1}(\mathcal{T}), &\text{if} & k>0, \\
                          \Upsilon_{k+1}...\Upsilon_0(\mathcal{T}), &\text{if} & k<0,
                         \end{array}\right.\mathcal{T}^{(k)} \in \mathfrak{X}_{\phi^k\zeta}.
$$

The space of tilings $i \in \llbracket1,m\rrbracket$ by tiles of order $k$ is essential for the introduction of finitely additive measures in the next section.

Let $\zeta$ be a primitive tile substitution in $\R^d$ over the finite set of tiles $\mathcal{A}=\{T_1,...,T_m\}$ and $S$ its incidence matrix. We denote by $\theta_1> |\theta_2|\geq...\geq |\theta_m|$ the eigenvalues of the matrix $S^t$ ($\theta_1$ being real, greater than one, and of multiplicity one by Perron-Frobenius Theorem). Remark that $\R^d$ acts on $\mathfrak{X}_{\zeta}$ by translation. For any vector $x \in \R^d$ let $h_x$ be the map from $\mathfrak{X}_{\zeta}$ to itself that sends a tiling $\mathcal{T}$ to its translate by vector $-x$ i.e. $\mathcal{T}-x$. This defines a continuous action of $\R^d$ on $\mathfrak{X}_{\zeta}$ by translation.

 This defines a dynamical system $(\mathfrak{X}_{\zeta},\R^d)$. Remark also that iterating $\zeta$ over a certain tile may converge towards a tiling of the space $\R^d$. This tiling is of course in the tiling space $\mathfrak{X}_{\zeta}$.

As already stated in Remark \ref{r.primitive}, it is usual to only consider the case of primitive tile substitutions, which is what is done in this article, but we also add the classical property of finite local complexity.

\begin{definition}
 A tiling $\mathcal{T}$ has finite local complexity (FLC) if, for any positive real number $r$, there is only a finite number of different patches of diameter at most $r$ in $\mathcal{T}$ up to translations.
 
 A tile substitution $\zeta$ is FLC if every tiling in the tiling space $\mathfrak{X}_{\zeta}$ has finite local complexity.
\end{definition}

\begin{remark}
 For the most part, classical self-similar tilings defined by tile substitutions are FLC. For an example of non-FLC tile substitution we refer the reader to \cite{kenyon_nonflc} for instance.
\end{remark}

\begin{theorem}{\cite{robinson}}
 If $\zeta$ is a primitive substitution with finite local complexity then $(\mathfrak{X}_{\zeta},\R^d)$ is uniquely ergodic. We denote by $\mu$ the unique ergodic probability measure of the dynamical system $(\mathfrak{X}_{\zeta},\R^d)$.
\end{theorem}

This theorem has as an immediate consequence the following lemma.

\begin{lemma}{\cite{robinson}}\label{l.mu_zeta_inv}
 If $\zeta$ is a primitive substitution with finite local complexity, then $\mu$ is invariant under the substitution action, i.e. $\zeta_*\mu:=\mu\circ \zeta^{-1}=\mu$.
\end{lemma}

\begin{proof}
 It is an easy check that the measure $\zeta_*\mu$ is invariant under the action of $\R^d$ since $\zeta$ acts on vectors as an homothety of ratio $\lambda$ and $\mu$ is invariant under the action of $\R^d$ by translation. Hence, by unique ergodicity, $\zeta_*\mu=\mu$.
\end{proof}

\section{Finitely-additive measures and Ergodic deviation}\label{s.deviation}

We now define finitely-additive measures as done in \cite{limitthm}.

\begin{definition}
 The rapidly expanding subspace $E^{++}$ is the linear span of Jordan cells of $S^t$ whose eigenvalues $\theta$ satisfy
 $$
 |\theta|>\theta_1^{\frac{d-1}{d}}=\lambda^{d-1}.
 $$
 \end{definition}
In a similar way, denote by $E^+$ the linear span of Jordan cells of $S^t$ whose eigenvalues are greater than 1. Equivalently, denote $\widetilde E^{++}$ the rapidly expanding subspace of $S$.

 We define, for any vector $v \in E^+$ and any tiling $\mathcal{T}\in\mathfrak{X}_{\zeta}$ on tiles of order $k\in \Z$, a finitely-additive $\Phi^+_{v,\mathcal{T}}$ measure as follows:
$$
 \Phi^+_{v,\mathcal{T}}(\mathrm{supp}(T))=((S^t)^k v)_j,\ \text{if}\  \exists k \in \Z, \ y \in \R^d \:\ T= \phi^k(T_j)-y\in \mathcal{T}^{(k)}.
$$

The definition of such finitely additive measures can be extended to Lipschitz domains (that is to say open sets of $\R^d$ whose boundaries are locally graphs of continuous Lipschitz functions). This is done in \cite{limitthm}.

The following lemma gives a ``self-similarity'' property for finitely-additive measures after their extension to any Lipschitz domains.

\begin{lemma}{\cite[Lemma 3.2]{limitthm}}\label{l.selfsimilar}
For any $v \in E^{++}$, there exists finitely-additive measures $\Phi_{v,\mathcal{T}}^+$ defined on the ring of sets generated by Lipschitz domain in $\R^d$. Moreover, if  $S^tv=\theta v$, they satisfy, for any  Lipschitz domain $\Omega$, the following:
$$
\quad \Phi_{v,\zeta(\mathcal{T})}^+(\phi(\Omega))=\Phi_{\theta v,\mathcal{T}}^+(\Omega)=\theta \Phi_{v,\mathcal{T}}^+(\Omega).
$$
 \end{lemma}
 
\begin{lemma}\cite[Lemma 3.3]{limitthm}\label{l.meas_bound}
Suppose that $v \in E^{++}$ belongs to the $S^t$ invariant subspace corresponding to a Jordan block of size $s\geq1$ with eigenvalue $\theta$ and such that $\|v\|=1$. Then for a Lipschitz domain $\Omega$, there exists a positive constant $C_1$ such that, for any positive real number $R$ (writing $\Omega_R=R\Omega$),
$$
\quad \left| \Phi_{v,\mathcal{T}}^+(\Omega_R) \right| \leq C_1(\log R)^{s-1}R^{\alpha}, \quad \text{where} \quad \alpha=\frac{d \log |\theta|}{\log \theta_1}.
$$
\end{lemma}

Before stating the main theorem on the estimates of deviation in ergodic sums, let us remark that it only gives estimates for cylindrical functions which we define as follows.

\begin{definition}\label{d.cylindrical}
A function $f:\mathfrak{X}_{\zeta}\rightarrow \R$ in $L^1(\mathfrak{X}_{\zeta},\mu)$ is cylindrical if it depends only on the tile containing the origin. More precisely, there is a family of functions $(\varPsi_i)_{i \in \llbracket1,m\rrbracket}$ where \mbox{$\varPsi_i: \mathrm{supp}(T_i)\rightarrow \R$} and $\varPsi_i \in L^1 (\mathrm{supp}(T_i), \text{Leb}_d)$, where $\text{Leb}_d$ is the Lebesgue measure on $\R^d$, such that
$$
f(\mathcal{T})=\varPsi_i(x)\ \text{if} \ 0\in \mathrm{supp}(T_i)-x,\  T_i-x\in \mathcal{T}.
$$
\end{definition}

\begin{definition}\label{d.mphimoins}
 For any cylindrical function $f$ and for any vector $v=(v_1,...,v_m)\in \C^m$ we set
 $$
 m_{\Phi_{v}^-}(f)=\sum_{i=1}^m v_i \int_{\mathrm{supp}(T_i)}\varPsi_i(x)dx.
 $$
 
\end{definition}

Now denote by $l$ the dimension of $E^{++}$ (and $\widetilde E^{++}$) and choose bases $(u_i)_{i\leq l}$ and $(\widetilde u_i)_{i\leq l}$ of $E^{++}$ and $\widetilde E^{++}$ respectively. We write  $\Phi_{i,\mathcal{T}}^+= \Phi_{u_i,\mathcal{T}}^+$ and $m_{\Phi_{i}^-}=m_{\Phi_{\widetilde u_i}^-}$.

We can state the theorem describing the deviation of ergodic averages for the dynamical system $(\mathfrak{X}_{\zeta},\R^d,\mu)$ on Lipschitz domains like thus.

\begin{theorem}\cite[Theorem 4.3]{limitthm}\label{t.erg_dev}
Let $\zeta$ be a primitive substitution with finite local complexity such that the dynamical system $(\mathfrak{X}_{\zeta},\R^d)$ is not periodic, and let $\Omega$ be a bounded Lipschitz domain in $\R^d$. Then there exists a positive constant $C$ depending on $\zeta$ and $\Omega$ such that, for any cylindrical function $f$, any tiling $\mathcal{T} \in \mathfrak{X}_{\zeta}$ and any $R>0$,
 $$\left| \int_{\Omega_R}f\circ h_x(T)dx-\text{Leb}_d(\Omega_R)\int_{\mathfrak{X}_{\zeta}}fd\mu - \sum_{n=2}^l \Phi_{n,\mathcal{T}}^+(\Omega_R)\cdot m_{\Phi_{n}^-}(f) \right|\leq CR^{d-1}(\log R)^{s}\int_{\mathfrak{X}_{\zeta}}|f|d\mu$$
 where $s$ is the maximal size of the Jordan block corresponding to eigenvalues satisfying $|\theta| = \theta_1^{\frac{d-1}{d}}$. If there are no such eigenvalues, then $s=0$.
\end{theorem}

\section{Spectral measures at zero}\label{s.spectral_measures}

Recall that for a measure $m$ on a space $\R^d$, the Fourier transform $\widehat{m}$ of $m$ is a function from $\R^d$ to $\C$ defined by
$$
\forall x \in \R^d,\  \widehat{m}(x)=\int_{\omega \in\R^d}e^{2i\pi\langle\omega,x\rangle}dm(\omega).
$$

We start by stating the spectral theorem:

\begin{theorem}\label{t.spectral}
 Let $f \in L^2(\mathfrak{X}_{\zeta},\mu)$. There exists a positive measure $\sigma_f$ on $\R^d$ called the spectral measure defined by $ \widehat{\sigma_f}(x) := \langle f \circ h_x, f \rangle $ for all $x \in \R^d$
such that the function mapping the function  $\left(\mathcal{T}\mapsto f\circ h_x (\mathcal{T})\right)$ in $ L^2(\mathfrak{X}_{\zeta},\mu)$ to the function $\left(\omega\mapsto e^{2i\pi \langle \omega,x \rangle}\right)$ extends to an isometry $J$ between a closed subspace of $L^2(\mathfrak{X}_{\zeta},\mu)$ and $L^2(\R^d,\sigma_f)$.
\end{theorem}

We refer the reader to \cite{ktspectral} for details on this theorem, but in particular it states that
$$
 \forall x \in \R^d, \int_{\omega \in {\R}^d} e^{2i\pi\langle \omega , x \rangle} d\sigma_f(\omega) = \langle f \circ h_x, f \rangle,
$$
where
$$
\langle f \circ h_x, f \rangle=\int_{\mathcal{T}\in \mathfrak{X}_\zeta} f(\mathcal{T})f\circ h_x(\mathcal{T})d\mu(\mathcal{T}).
$$
This theorem introduces the spectral measures on which Theorem \ref{t.main} gives results. In order to prove Theorem \ref{t.main} we also need the set
$$
\mathcal{S}_F(\R_+)=\{f\in C^{\infty}(\R_+, \R) \mid \forall \alpha, \beta \in \N, \sup \limits_{x \in \R_+}|x^{\alpha} f^{(\beta)}(x)| < \infty \quad \text{and}\quad \forall k >0, f^{(k)}(0)=0\}
$$
and the continuous function $\text{Rad}: \mathcal{S}_F(\R_+) \rightarrow \mathcal{S}(\R^d)$ (where $\mathcal{S}(\R^d)$ is the Schwartz space of rapidly decreasing functions on $\R^d$) defined by
$$
\forall \psi \in \mathcal{S}_F(\R_+), \forall x \in \R^d,\quad \text{Rad}(\psi)(x)=\psi(\|x\|).
$$

\subsection{Proof of Theorem \ref{t.main}}

\subsubsection{Step 1: using the spectral isomorphism}

Recall that Theorem \ref{t.spectral} states that
$$ \forall x \in \R^d, \int_{\omega \in {\R}^d} e^{2i\pi\langle \omega , x \rangle} d\sigma_f(\omega) = \langle f \circ h_x, f \rangle.$$

Moreover, the spectral isomorphism $J$ satisfies $J^{-1}\left(\omega \mapsto e^{2i\pi\langle \omega , x \rangle}\right)=\left( \mathcal{T} \mapsto f \circ h_x(\mathcal{T}) \right)$ for all $x$ in $\R^d$.

Let us analyse the behaviour of $\sigma_f$ near zero. In order to do this, we fix a function $ \psi \in \mathcal{S}_F(\R_+) $ and study the integral $\int | \text{Rad}(\psi)(R\omega)|^2d\sigma_f(\omega)$ for all $R>0$.

Applying the Inverse Fourier Transform, which preserves $\mathcal{S}(\R^d)$ and the radial property of $\text{Rad}(\psi)$, we obtain

$$\text{Rad}(\psi)(R\omega)=R^{-d}\int_{x \in {\R}^d} e^{2i\pi\langle \omega , x \rangle}\widehat{\text{Rad}(\psi)}(x / R) dx.$$

By the spectral isomorphism, this function (of $\omega$) maps, by $J^{-1}$, to
$$
R^{-d}\int_{x \in {\R}^d} f\circ h_x(\mathcal{T}) \cdot \widehat{\text{Rad}(\psi)}(x / R) dx
$$
which is a function of $\mathcal{T}$ in $L^2(\mathfrak{X}_\zeta, \mu)$. Since the spectral isomorphism preserves the $L^2$ norm, we obtain

$$\int_{\omega \in {\R}^d} |\text{Rad}(\psi)(R\omega)|^2 d\sigma_f(\omega)=
R^{-2d}\int_{\mathcal{T} \in \mathfrak{X}_\zeta}\left|\int_{x \in {\R}^d} f\circ h_x(\mathcal{T}) \cdot \widehat{\text{Rad}(\psi)}(x /R) dx \right| ^2 d\mu(\mathcal{\mathcal{T}}).$$

Now let us change coordinates of the right-hand side and use spherical coordinates. Let $u$ be the mapping from spherical coordinates to euclidian coordinates on $\R^d$. Since Fourier transform preserves radial functions, $\widehat{\text{Rad}(\psi)}(\rho,0,...0)=\widehat{\text{Rad}(\psi)}(u(\rho,\phi_1,...\phi_{d-1}))$. For simplicity, in all that follows, let us write $\widehat{\text{Rad}(\psi)}(\rho)$ instead of $\widehat{\text{Rad}(\psi)}(\rho,0,...0)$. Then the integral 
$$
\int_{x \in {\R}^d} f\circ h_x(\mathcal{T}) \cdot \widehat{\text{Rad}(\psi)}(x /R) dx
$$ can be written as


$$\int_{\rho=0}^{+\infty}\int_{[0,\pi]^{d-1}} \int_{\phi_{d-1} =0}^{2\pi} 
f\circ h_{u(\rho,\phi_1,...\phi_{d-1})}(\mathcal{T}) \cdot
\widehat{\text{Rad}(\psi)}(\rho/R) \cdot 
\rho^{d-1}\prod_{j=1}^{d-2}(\sin^{d-1-j}(\phi_j))  d\phi_{d-1}...d\phi_1 d\rho$$
or
$$
\int_{\rho=0}^{+\infty}\widehat{\text{Rad}(\psi)}(\rho/R) \int_{[0,\pi]^{d-1}} \int_{\phi_{d-1} =0}^{2\pi} 
f\circ h_{u(\rho,\phi_1,...\phi_{d-1})}(\mathcal{T}) 
\rho^{d-1}\prod_{j=1}^{d-2}(\sin^{d-1-j}(\phi_j))  d\phi_{d-1}...d\phi_1 d\rho
$$
Now, an integration by parts with regards to the variable $\rho$, differentiating $\widehat{\text{Rad}(\psi)}(\rho/R)$ and integrating $ \int_{[0,\pi]^{d-1}} \int_{\phi_{d-1} =0}^{2\pi} 
f\circ h_{u(\rho,\phi_1,...\phi_{d-1})}(\mathcal{T}) 
\rho^{d-1}\prod_{j=1}^{d-2}(\sin^{d-1-j}(\phi_j))  d\phi_{d-1}...d\phi_1$, yields

\begin{equation}\label{e.sftrho}
 \int_{x \in {\R}^d} f\circ h_x(\mathcal{T}) \cdot \widehat{\text{Rad}(\psi)}(x/R) dx = 
 2i\pi R^{-1} \int_{\rho=0}^{+\infty} S(f,\mathcal{T},\rho)\mathcal{F}(x_1 \text{Rad}(\psi)(x))(\rho/R)d\rho,
\end{equation}
 where 
 
 $$ S(f,\mathcal{T},\rho)=\int_{x \in B(0,\rho)} f \circ h_x(\mathcal{T})dx.$$
 
 \subsubsection{Step 2: estimating the rest}
 
 From Theorem \ref{t.erg_dev} and Lemma \ref{l.meas_bound}, and since $\int f d\mu=0$, we have
 
 $$ S(f,\mathcal{T},\rho)= \Phi_{2,\mathcal{T}}^+(\rho) \cdot m_{\Phi_2^-}(f) + \mathcal{R}(\rho),$$
 where
 $$|\mathcal{R}(\rho)| \leq C_1 \max(1,\rho^{\alpha - \varepsilon})$$
 with $\alpha = \frac{d \log (\theta_2)}{\log (\theta_1)}$ and some $\varepsilon > 0$ small enough.
 
 Since $\mathcal{F}(x_1 \text{Rad}(\psi)(x))$ is in $\mathcal{S}(\R^d)$, we have
 
 $$|\mathcal{F}(x_1 \text{Rad}(\psi)(x))(\rho)| \leq C_{\psi,\alpha} \min (1,\rho^{-\alpha-1}).$$
 
 Finally, 
$$
2\pi R^{-1}\left| \int_{\rho =0}^{+\infty} \mathcal{R}(\rho)\mathcal{F}(x_1 \text{Rad}(\psi)(x))(\rho)d\rho \right|
\leq
C_1C_{\psi,\alpha}2\pi R^{-1}\left(1+\int_1^R \rho^{\alpha - \varepsilon} d\rho + \int_R^{+\infty} \rho^{\alpha-\varepsilon}(\rho/R)^{-\alpha -1}d\rho\right)
$$
hence, since $\alpha$ is in $(d-1,d)$,
\begin{equation}\label{e.rest}
 2\pi R^{-1}\left| \int_{\rho =0}^{+\infty} \mathcal{R}(\rho)\mathcal{F}(x_1 \text{Rad}(\psi)(x))(\rho)d\rho \right|
=O(R^{\alpha- \varepsilon}).
\end{equation}

\subsubsection{Step 3: change of variables in the main term}

Now that the error term in $\int_{\rho=0}^{+\infty} S(f,\mathcal{T},\rho)\mathcal{F}(x_1 \text{Rad}(\psi)(x))(\rho/R)d\rho$ is estimated, let us study the main term in Equation~(\ref{e.sftrho}). To that end, we will assume that $$R=\lambda^N=\theta_1^{\frac{N}{d}}, N\geq1.$$

Let  $\rho=\lambda^N r = R r$, and let us use Lemma \ref{l.selfsimilar} to do the following renormalisation:
$$\Phi_{2,\mathcal{T}}^+(\rho) = \Phi_{2,\mathcal{T}}^+(\lambda^N r)=\theta_2^N \Phi_{2,\zeta^{-N}(\mathcal{T})}^+(r)=R^{\alpha} \Phi_{2,\zeta^{-N}(\mathcal{T})}^+(r).$$

After this change of variables we have
\begin{align*}
   &2i \pi R^{-1}\int_{\rho=0}^{+\infty}  (\Phi_{2,\mathcal{T}}^+(\rho) \cdot m_{\Phi_2^-}(f) + \mathcal{R}(\rho) )\mathcal{F}(x_1 \text{Rad}(\psi)(x))(\rho/R)d\rho\\
   &=2i \pi R^{\alpha} m_{\Phi_2^-}(f) \int_{r=0}^{+\infty} \Phi_{2,\zeta^{-N}(\mathcal{T})}^+(r)\mathcal{F}(x_1 \text{Rad}(\psi)(x))(r)dr.
  \end{align*}

The estimation of the rest in Equation \ref{e.rest} and this last equality can be used to get that 

\begin{align*}
 &\int_{\omega \in {\R}^d} |\text{Rad}(\psi)(R\omega)|^2 d\sigma_f(\omega)\\
&=4\pi^2 R^{2\alpha - 2d} (m_{\Phi_2^-}(f))^2 
\int_{\mathcal{T}\in\mathfrak{X}_{\zeta}} 
\left|\int_{r=0}^{+\infty} \Phi_{2,\mathcal{T}}^+(r)\mathcal{F}(x_1 \text{Rad}(\psi)(x))(r)dr\right|^2d\mu(\mathcal{T}) + O(R^{2\alpha - 2d -2\varepsilon})
\end{align*}
 
as $R\rightarrow + \infty$, since $\mu$ is $\zeta$ invariant by Lemma \ref{l.mu_zeta_inv}. Finally, with $R= \lambda^N$

\begin{align*}
 &\lim_{N\rightarrow +\infty} \frac{\int_{\omega \in {\R}^d} |\text{Rad}(\psi)(\lambda^N\omega)|^2 d\sigma_f(\omega)}{\lambda^{N(2\alpha - 2d)}}\\
 &=(2\pi m_{\Phi_2^-}(f))^2
\int_{\mathcal{T}\in\mathfrak{X}_{\zeta}} 
\left|\int_{r=0}^{+\infty} \Phi_{2,\mathcal{T}}^+(r)\mathcal{F}(x_1 \text{Rad}(\psi)(x))(r)dr\right|^2d\mu(\mathcal{T}).
\end{align*}

We can remark that the right-hand side is non zero by the hypothesis on $m_{\Phi_2^-}(f)$ and because $ \Phi_{2,\mathcal{T}}^+$ is not almost everywhere zero, since our substitution is defined on polyhedral prototiles (see Lemma 3.5 and section 6.2 of \cite{limitthm} for more details), and because this is true for all choice of $\psi$.

\subsubsection{Step 4: regularity of the limit}

This limit is a quadratic form on $\mathcal{S}_{F}(\R_+)$. Let us write the associated bilinear form on the space $\mathcal{S}_F(\R_+)\times\mathcal{S}_F(\R_+)$:

\begin{align*}
&Q(\psi_1,\psi_2):=\lim_{N\rightarrow +\infty} \frac{\int_{\omega \in {\R}^d} \text{Rad}(\psi_1)(\lambda^N\omega)\text{Rad}(\psi_2)(\lambda^N\omega) d\sigma_f(\omega)}{\lambda^{N(2\alpha - 2d)}}
\\&=(2 \pi m_{\Phi_2^-}(f))^2
\int_{\mathcal{T}\in\mathfrak{X}_{\zeta}} 
\left(\int_{r=0}^{+\infty} \Phi_{2,\mathcal{T}}^+(r)\mathcal{F}(x_1 \text{Rad}(\psi_1)(x))(r)dr \int_{r=0}^{+\infty} \Phi_{2,\mathcal{T}}^+(r)\mathcal{F}(x_1 \text{Rad}(\psi_2)(x))(r)dr\right)
d\mu(\mathcal{T}).
\end{align*}

 Let us remark that this bilinear form (which depends solely on the product $\psi_1 \psi_2$) is continuous, non-negative on non-negative functions, and  not identically zero.
 
Let us also prove the following.

\begin{lemma}
 Every function in $C_0^{\infty}(\R_+,\R)$ is the uniform limit of a sequence of functions in $\mathcal{S}_F(\R_+)$.
\end{lemma}

\begin{proof}
  To see this, it suffices to take a function $f$ in $C_0^{\infty}(\R)$ and a sequence of functions $(f_n)_{n\in \N^*}$ in $C(\R)$ such that

\begin{itemize}
 \item $\forall n \in \N^*, \forall x \in [\frac{-1}{n},\frac{1}{n}], f_n(x)=f(0)$;
 \item $(f_n)_{n\in \N^*}$ converges uniformly towards $f$;
 \item $\forall n \in \N^*, f_n$ is compactly supported.
\end{itemize}

Then we take a regularizing sequence $\left( \rho_m \right)_{m \in \N}$ and remark that for all $n \in \N^*$ and $m>n$, the convolution $\rho_m * f_n$ is in $C^{\infty}_0(\R)$ and is flat in 0. Thus there exists a strictly increasing sequence $(m_n)_{n \in \N^*} \in \N^*$ such that the sequence $(\rho_{m_n} * f_n)_{n\in \N^*}$ (consisting of $C^{\infty}_0$ functions flat in 0) uniformly converges to $f$, thus proving the lemma.
\end{proof}

Let us also remark that

$$\forall \psi_1, \psi_2 \in \mathcal{S}_F(\R_+),  \quad Q(\psi_1^2,\psi_2) < Q(\psi_1,\psi_1) \|\psi_2\|_{\infty}.$$

As a consequence, for all $\psi_1$ in $\mathcal{S}_F(\R_+)$, $\psi_2 \mapsto Q(\psi_1^2,\psi_2)$ has a continuous extension with domain the whole of $C^{\infty}_0(\R_+, \R)$. Thus it is a distribution that is non-negative on non-negative functions.

\subsubsection{Step 5: a technical lemma}

 Let us prove the following lemma.

\begin{lemma}
 There exists a $\sigma$-finite positive Radon measure $\eta$ on $\R_+$ such that for any $\psi_1,\psi_2$ in $C_0^{\infty}(\R_+, \R)$ we have
 $$ Q (\psi_1,\psi_2)=\int_{\R_+}\psi_1\psi_2d\eta.$$
\end{lemma}

\begin{proof}

In order to prove this lemma, we need to use a classical result from the theory of distribution stating that if a distribution $D$ supported on $\R_+$ is non-negative on non-negative functions in $C_0^{\infty}(\R_+, \R)$ then there exists a $\sigma$-finite positive Radon measure $\eta$ supported on $\R_+$ such that 
$$D(\varphi)=\int_{\R_+} \varphi d\eta.$$

First we observe that for any non-negative $\psi_1 \in \mathcal{S}_F(\R_+)$, the functional $\varphi \mapsto Q(\psi_1^2, \varphi)$ is a distribution on $\R_+$ that is non-negative on $\varphi \geq 0$. Thus there exists a measure $\eta_{\psi_1}$ such that

$$\forall \varphi \in C_0^{\infty}(\R_+, \R), \quad Q(\psi_1^2,\varphi)=\int_{\R_+}\varphi d\eta_{\psi_1}.$$

Remark that $Q((\psi_1\psi_2)^2,\varphi)=Q(\psi_1^2,\psi_2^2\varphi)$, hence

$$d\eta_{\psi_1\psi_2}=\psi_2^2d\eta_{\psi_1}\quad \text{for}\quad \psi_1,\psi_2 \in \mathcal{S}_F (\R_+).$$

Thus, taking any $\psi >0$, $d\eta=\frac{1}{\psi^2}d\eta_{\psi}$ does not depend on the choice of $\psi$.

Now let $\psi_1,\psi_2$ be two functions in $C_0^{\infty}(\R_+, \R)$ and remark

$$Q(\psi_1,\psi_2):=Q(\psi^2,\frac{\psi_1\psi_2}{\psi^2})=\int_{\R_+}\psi_1\psi_2d\eta,$$
which proves the lemma.
 
\end{proof}

Let us now conclude that the following formula

$$\lim_{N\rightarrow +\infty} \frac{\int_{\omega \in {\R}^d} \text{Rad}(\psi_1)(\lambda^N\omega)\text{Rad}(\psi_2)(\lambda^N\omega) d\sigma_f(\omega)}{\lambda^{N(2\alpha - 2d)}}=\int_{\R_+}\psi_1\psi_2d\eta$$

also holds for characteristic functions of intervals $[0,a]$ where $a$ is a point of continuity of the measure $\eta$.

Indeed, let us choose $a\in \R^+ $ such that $\eta(\{a\})=0$, let $\psi=\psi_1=\psi_2=\chi_{[0,a]}$ and choose sequences of functions $\psi_n^+, \psi_n^-$ in $C_0^{\infty}(\R_+,\R)$, flat in $0$, approximating $\psi$ respectively from above and below, converging pointwise to $\psi$ and uniformly on the complement of $(a-\delta, a+\delta)$ for any $\delta >0$.

By continuity of $\eta$ over non-atomic $a$ we have

$$\forall \varepsilon>0,\quad \exists \delta>0, \quad \eta((a-\delta, a+\delta))<\varepsilon.$$

As a consequence, 

$$\lim_{N\rightarrow +\infty} \int_{\R_+}\left(\psi_n^+\psi_n^+-\psi_n^-\psi_n^-\right)d\eta=0,$$

and so the formula is still true for the characteristic function $\chi_{[0,a]}$ whenever $\eta(\{a\})=0$. This proves Theorem \ref{t.main}.

\end{document}